\documentclass[11pt]{article}
\usepackage{geometry}
\usepackage{color,graphicx}
\usepackage{amssymb}
\usepackage{amsmath}
\usepackage{amsfonts}
\usepackage{graphicx}
\usepackage{amsthm}
\usepackage{enumerate}
\usepackage[mathscr]{eucal}
\usepackage{mathrsfs}
\usepackage{verbatim}
\usepackage{yhmath}
\usepackage{bbm}
\usepackage[normalem]{ulem}
\usepackage{dsfont}
\usepackage{xcolor}

\usepackage{hyperref}
\hypersetup{colorlinks, linkcolor=blue}
\usepackage{soul}
\soulregister\cite7
\soulregister\eqref7
\soulregister\eqref7
\hypersetup{colorlinks, linkcolor=blue}

\newtheorem{thm}{Theorem}[section]
\newtheorem{lemma}[thm]{Lemma}

\newtheorem{thmx}{Theorem}

\theoremstyle{definition}

\newtheorem{remark}[thm]{Remark}

\newcommand\ddd{\mathrm{d}}

\newcommand\supp{\mathrm{supp}}
\newcommand\holder{Hölder's~}

\newcommand\bR{\mathbb{R}}

\newcommand\bC{\mathbb{C}}

\newcommand{\eps}{\varepsilon}

\def \l {\left}
\def \r {\right}

\newcommand\nnfootnote[1]{   
\begin{NoHyper}
\renewcommand\thefootnote{}\footnote{#1}
\addtocounter{footnote}{-1}
\end{NoHyper}}

\title{Extremizers for the Strichartz Inequality for a Fourth-Order Schrödinger Equation}
\author{Boning Di and Ryan Frier}
\date{ }

\begin{document}

\maketitle
\nnfootnote{{2020 \emph{Mathematics Subject Classification}: Primary 42B10; Secondary 35B38, 35Q41.}}
\nnfootnote{\emph{Key words and phases}: Sharp Fourier restriction theory, extremizers, fourth-order Schr\"odinger equation.}

\begin{abstract}
    In this paper, we consider the Strichartz inequality for a fourth-order Schrödinger equation on $\bR^{2+1}$.  We show that extremizers exist using a linear profile decomposition which follows from the endpoint version decomposition and the stationary phase method. Based on the existence of extremizers, we investigate the associated Euler-Lagrange equation to show that the extremizers have exponential decay and consequently must be analytic.
\end{abstract}

\section{Introduction}
In this paper, we investigate the Strichartz inequality for the fourth-order Schrödinger equation as follows:
\begin{equation}\label{E:4-Strichartz}
\l\|[e^{it|\nabla|^{4}}]u_0\r\|_{L_{t,x}^6(\mathbb{R}^3)}\leq \mathbf{S} \|u_0\|_{L_x^2(\bR^2)},
\end{equation}
where
\[\mathbf{S}:=\sup\l\{\l\|[e^{it|\nabla|^4}]u\r\|_{L_{t,x}^{6}(\bR^{3})}: \|u\|_{L_x^2(\bR^2)}=1\r\}\]
is the \textit{sharp constant} and
\[[e^{it|\nabla|^4}]u(x)=e^{it\Delta^2}u(x):=\mathscr{F}^{-1}e^{it|\xi|^4}\mathscr{F}u(x), \quad  \mathscr{F}u(\xi):=\int_{\bR^2} e^{-ix\xi}u(x)\ddd x,\]
and $x \in \bR^2$.  Here $x\xi:=x_1\xi_1+x_2\xi_2$ for vectors $x=(x_1,x_2)$ and $\xi=(\xi_1,\xi_2)$ in $\bR^2$. This Strichartz-type estimate \eqref{E:4-Strichartz} can also be seen as some \textit{Fourier extension estimate}, since the space-time Fourier transform of $[e^{it\Delta^2}]u$ is supported on the surface $(\xi,|\xi|^4)$ in $\bR^{3}$.  To see a proof of \eqref{E:4-Strichartz}, see Kenig, Ponce, and Vega's work \cite[Theorem 3.1]{KPV1991}.

Our first result states that there exists a function that makes inequality \eqref{E:4-Strichartz} an equality. Such functions will be called \textit{extremizers} for $\mathbf{S}$. We do this by constructing a linear profile decomposition without a frequency translation parameter, which will be an adaptation of the proof of \cite[Theorem 2.4]{HS2012} albeit with a small twist since our case is $d=2$. Then following some standard arguments which can be found in Han's work \cite{Han2015} as well as Hundertmark and Shao's work \cite{HS2012}, this linear profile decomposition will directly imply the desired existence of extremizers for inequality \eqref{E:4-Strichartz}.

To show that extremizers are analytic, we start by showing that $e^{\mu|\xi|^4}\widehat{f} \in L^2(\bR^2)$ for some $\mu>0$.  Through the dominated convergence theorem, this implies that $f$ is analytic.  This proof closely follows that of Brocchi, Oliveira, and Quilodrán in \cite{BOQ2020}, as well as Jiang and Shao in \cite{JS2016}.

In summary, our main result is the following theorem:
\begin{thm} \label{T:main}
	There exist extremizers for the sharp constant $\mathbf{S}$; furthermore, if $f_0$ is an extremizer, then $f_0$ can be extended to an entire function on $\bC^2$.
\end{thm}

The history of this problem is not as rich as more common topics (such as the standard Schrödinger equation or the Airy equation), but it has been studied nevertheless. Inequality \eqref{E:4-Strichartz} can be understood as an investigation of the solution to
\begin{equation}
    \begin{split}
        \label{E:4-Schrod}
        iu_t-\Delta^2u&=0; \\
        u(x, 0)&=u_0(x) \in L^2(\bR^d).
    \end{split}
\end{equation}
A more general version of this initial value problem is
\begin{equation*}
    \begin{split}
        \label{E:Gen-Schrod}
        iu_t-\left|\nabla \right|^\alpha u&=0 \\
        u(x, 0)&=u_0(x) \in L^2(\bR^d),
    \end{split}
\end{equation*}
where $\alpha \geq 2$.  The topic is very well known, and is still widely discussed when $\alpha=2$.  Foschi found functional equations related to the Strichartz inequality in \cite{Foschi2007}, and used these to find precise values for sharp constants when $d=1, 2$.  Similarly, Foschi also used these functional equations to show that extremizers are Gaussians. In \cite{HZ2006} Hundertmark and Zharnitsky found a new representation of the Strichartz inequality:
\begin{equation*}
    \begin{split}
        \int_\bR\int_\bR\left|e^{it\Delta }f(x) \right|^6 \ddd x\ddd t&= \frac{1}{2\sqrt{3}} \langle f\otimes f \otimes f, P_1(f\otimes f \otimes f)\rangle_{L^2(\bR^3)}; \\
        \int_\bR\int_{\bR^2}\left|e^{it\Delta }f(x) \right|^4 \ddd x\ddd t&= \frac{1}{4} \langle f\otimes f, P_2(f\otimes f) \rangle_{L^2(\bR^4)},
    \end{split}
\end{equation*}
where $P_i$ are orthogonal projection operators and $f\otimes g(x_1, x_2)=f(x_1)g(x_2)$ is the standard tensor product.  Using this, they were also able to obtain that extremizers were Gaussians. In \cite{JS2016} Jiang and Shao showed that when $(\alpha,d)=(2,2)$, then extremizers must be analytic, and used the functional equations in Foschi's paper \cite{Foschi2007} to give an alternative proof to show that extremizers are Guassians. Frier and Shao showed that extremizers must be Gaussians for the case when $(\alpha,d)=(2,1)$ in \cite{FS2022DPDE} as well.  Then in \cite{BOQ2020}, Brocchi, Oliveira, and Quilodrán investigated more generally the Fourier extension operator $\int_{\bR}e^{i(xy+|y|^{\alpha}t)}|y|^{(\alpha-2)/6}f(y) \ddd y$ with $\alpha>1$ using some Br\'ezis-Lieb type lemma and geometric comparison principle.  As one can see, this corresponds to a more general $\alpha$ in dimension $d=1$. This general $(\alpha,1)$ case was also studied by Di and Yan in \cite{DY2023} using linear profile decomposition arguments. Furthermore, the existence of extremizers for one dimensional non-endpoint $\alpha$-Strichartz estimates is also deduced in \cite{DY2023}. Later, Di and Yan \cite{DY2024} have also established a precompactness result for high dimensional $(\alpha,d)$ cases with $\alpha\geq2$ by using Tao's bilinear restriction estimate and Lieb's missing mass method, as well as some multi-variable analysis such as oscillatory integral estimates. For the case $d=2$, based on the previous conclusions in \cite{OQ2018}, this precompactness result in \cite{DY2024} also implies some consequence on the existence of extremals for the endpoint $\alpha$-Strichartz inequalities. Here, for our main Theorem \ref{T:main}, we remark that
\begin{remark}
    Note that the recent work of \cite[Proposition 3.3]{DLY2024} has shown the endpoint profile decomposition for general index $\alpha$ and general dimension $d$. Following the outline of this paper, one should get similar conclusions for the general $(\alpha,d)$ case for $\alpha \in \mathbbm{N} \backslash \{1\}$.
\end{remark}

For the linear profile decomposition, there are abundant consequences for different kinds of equations and purposes, see, for instance, \cite{BG1999,BV2007,CK2007,Han2015,JPS2010,JSS2017,Keraani2001,MV1998,Shao2009}. We emphasize some consequences that are closely related to our paper. Recall the following family of fourth-order Schrödinger equations:
\begin{equation*}
	\begin{split}
		iu_t+\Delta^2 u-\mu \Delta u&=0, \qquad \mu \geq 0, \\
		u(0)&=u_0 \in L^2_x(\bR^d).
	\end{split}
\end{equation*}
Jiang, Pausader, and Shao \cite{JPS2010} established an associated linear profile decomposition for $d=1$ and then deduced a dichotomy result on the existence of extremizers for the endpoint Strichartz inequality. Later in \cite{JSS2017}, Jiang, Shao, and Stovall obtained some similar results for general dimensions. Recently, both of these dichotomy results were resolved by Brocchi, Oliveira, and Quilodrán in \cite{BOQ2020,OQ2018}. On the other hand, for the non-endpoint Strichartz inequality \eqref{E:4-Strichartz}, Han showed in \cite{Han2015} that extremizers exist through the use of a linear profile decomposition which does not include the frequency parameters. In section \ref{S:existence of extremizers}, we will eliminate the frequency parameters in the profile decomposition of \cite{JSS2017}, and therefore, the existence of extremizers can be obtained. Then in section \ref{S:Differentiability}, we will investigate the analytic property of the extremizers.

\section*{Acknowledgements}
The authors would like to thank Shuanglin Shao for his valuable input, as well as suggesting the problem to the authors, and thank Dunyan Yan for his valuable support and advice.  The authors would also like to thank René Quilodrán for pointing out an error in the original version of the paper, as well as providing additional sources for the problem. The first author was supported by China Postdoctoral Science Foundation [Grant Nos. GZB20230812 \& 2024M753436]. This research was completed when the first author was visiting the University of Kansas, whose hospitality is also appreciated. The second author received a scholarship from the Office of Graduate Studies at the University of Kansas to assist with this project, and would like to thank the donors who provided for the scholarship.

\section{Existence of Extremizers}\label{S:existence of extremizers}
This section is devoted to showing the existence of extremizers for $\mathbf{S}$ using profile decomposition arguments. Indeed, the existence of extremizers has been previously shown by Han in \cite{Han2015}. Han's paper is dependent on a linear profile decomposition in \cite{PS2010}. This result is shown using Littlewood-Paley square functions, as well as the Bernstein property. Here we provide a proof of the linear profile decomposition Theorem \ref{T:LPD} that is based on the result by Jiang, Shao, and Stovall in \cite{JSS2017}. Our result is shown via Sobelev embedding and the stationary phase method. Our proof can be seen below, and here we state \cite[Theorem 3.1]{JSS2017} for completeness.
\begin{thmx}[Endpoint profile decomposition \cite{JSS2017}]\label{T:JSS2017}
	Let $(u_n)$ be a bounded sequence in $L^2(\bR^2)$. Then up to sub-sequences, there exists a sequence of operators $([T_n^j])$ defined by
	\[[T_n^j]\phi(x):=[e^{-it_n^j\Delta^2}]\l[(h_n^j)^{-1} e^{i(x-x_n^j)\xi_n^j} \phi\l(\frac{x-x_n^j}{h_n^j}\r)\r]\]
	with $(h_n^j,x_n^j,\xi_n^j, t_n^j)\in\bR_{+}\times \bR^2\times \bR^2 \times \bR$ and a sequence of functions $\phi^j\in L^2(\bR^2)$ such that for every $J\geq1$, we have the profile decomposition
	\begin{equation*}
		u_n=\sum_{j=1}^{J} [T_n^j]\phi^j+\omega_n^J,
	\end{equation*}
	where $\lim_{n\to\infty}|h_n^j\xi_n^j|=\infty$ or $\xi_n^j\equiv0$, meanwhile this decomposition has the following properties: firstly the remainder term $\omega_n^J$ has vanishing Strichartz norm
	\begin{equation}\label{T:JSS2017-1}
		\lim_{J\to\infty} \lim_{n\to\infty} \l\|[D^{\frac{1}{2}}] [e^{it\Delta^2}]\omega_n^J\r\|_{L_{t,x}^{4}(\bR^3)}=0,
	\end{equation}
	where $[D^s]f:=\mathscr{F}^{-1}|\xi|^s \mathscr{F}f$; secondly for each $J\geq1$ we have the $L^2$-orthogonality
	\begin{equation}\label{T:JSS2017-2}
		\lim_{n\to\infty}\l[\|u_n\|_{L_x^2}^2-\l(\sum_{j=1}^J \|\phi^j\|_{L_x^2}^2\r)-\|\omega_n^J\|_{L_x^2}^2\r]=0;
	\end{equation}
	and moreover for each $J\geq1$ there holds the Strichartz-orthogonality
	\begin{equation}\label{T:JSS2017-3}
		\limsup_{n\to\infty} \l(\l\|\sum_{j=1}^J [D^{\frac{1}{2}}] [e^{it\Delta^2}][T_n^j]\phi^j\r\|_{L_{t,x}^4}^4 -\sum_{j=1}^J \l\|[D^{\frac{1}{2}}] [e^{it\Delta^2}][T_n^j]\phi^j\r\|_{L_{t,x}^4}^4\r)=0.
	\end{equation}
\end{thmx}
Based on this endpoint profile decomposition Theorem \ref{T:JSS2017}, we can establish the corresponding non-endpoint linear profile decomposition, which will then lead to the existence of extremizers for $\mathbf{S}$. Here, we prove the desired conclusions by adapting the arguments in \cite[Theorem 2.4]{HS2012} which deals with the one dimensional case $u\in L^2(\bR)$. While in our two dimensional case, due to the differences between $|\xi_n\xi|^2$ and $|\xi|^2|\xi_n|^2$, we need to further introduce some linear transformation $A_0$ which will be seen later in \eqref{E:LPD-2.5}.

\begin{thm}[Non-endpoint profile decomposition]\label{T:LPD}
	Let $(u_n)$ be a bounded sequence in $L^2(\bR^2)$. Then up to sub-sequences, there exists a sequence of operators $([T_n^j])$ defined by
	\[[T_n^j]\phi(x):=[e^{-it\Delta^2}]\l[(h_n^j)^{-1} \phi\l(\frac{x-x_n^j}{h_n^j}\r)\r]\]
	with $(h_n^j,x_n^j,t_n^j)\in\bR_{+}\times \bR^2 \times \bR$ and a sequence of functions $\phi^j\in L^2(\bR^2)$ such that for every $J\geq1$, we have the profile decomposition
	\begin{equation}\label{T:LPD-1}
		u_n=\sum_{j=1}^{J} [T_n^j]\phi^j+\omega_n^J,
	\end{equation}
	where the decomposition processes the following properties: firstly, the remainder term $\omega_n^J$ has vanishing Strichartz norm
	\begin{equation}\label{T:LPD-2}
		\lim_{J\to\infty} \lim_{n\to\infty} \l\|[e^{it\Delta^2}]\omega_n^J\r\|_{L_{t,x}^{6}(\bR^3)}=0;
	\end{equation}
	secondly for each $J\geq1$ we have the $L^2$-orthogonality
	\begin{equation}\label{T:LPD-3}
		\lim_{n\to\infty}\l[\|u_n\|_{L^2}^2-\l(\sum_{j=1}^J \|\phi^j\|_{L_x^2}^2\r)-\|\omega_n^J\|_{L_x^2}^2\r]=0;
	\end{equation}
	and moreover for each $J\geq1$ there holds the Strichartz-orthogonality
	\begin{equation}\label{T:LPD-4}
		\limsup_{n\to\infty} \l(\l\|\sum_{j=1}^J [e^{it\Delta^2}][T_n^j]\phi^j\r\|_{L_{t,x}^6}^6 -\sum_{j=1}^J \l\| [e^{it\Delta^2}][T_n^j]\phi^j\r\|_{L_{t,x}^6}^6\r)=0.
	\end{equation}
\end{thm}
\begin{proof}[\textbf{Proof of Theorem \ref{T:LPD}}]
	Notice that Sobolev inequalities imply the following estimate:
	\[\l\|[e^{it\Delta^2}]\phi\r\|_{L_{t,x}^6(\bR^3)}\lesssim \l\|[D_t^{\frac{1}{12}}] [D_x^{\frac{1}{6}}] [e^{it\Delta^2}] \phi \r\|_{L_{t,x}^4(\bR^3)}= \l\|[D^{\frac{1}{2}}] [e^{it\Delta^2}] \phi\r\|_{L_{t,x}^4(\bR^3)}.\]
	Therefore the vanishing norm estimate \eqref{T:LPD-2} follows from the remainder term estimate \eqref{T:JSS2017-1} in Theorem \ref{T:JSS2017}. To eliminate the frequency parameters, as shown in the proof of \cite[Theorem 2.4]{HS2012}, the key point is to deduce the following estimate:
	\begin{equation}\label{E:LPD-1}
		\lim_{|\xi_n|\to\infty}\l\|[e^{it\Delta^2}][e^{i(\cdot)\xi_n}\phi]\r\|_{L_{t,x}^{6}}=0.
	\end{equation}
	Here, for convenience, we establish the desired estimate \eqref{E:LPD-1} with the notation $\xi_n$ instead of $h_n^j\xi_n^j$. Then the highly oscillatory terms, which mean the terms $[T_n^j]\phi^j(x)$ with
	\[\lim_{n\to\infty}|h_n^j\xi_n^j|=\infty,\]
	in Theorem \ref{T:JSS2017} can be reorganized into the remainder term based on this estimate \eqref{E:LPD-1}. After that, the desired Strichartz-orthogonality \eqref{T:LPD-4} of these profiles is much easier to establish due to the lack of frequency parameters, see also \cite[Lemma 2.7]{HS2012} for further details. The other conclusions come from Theorem \ref{T:JSS2017} accordingly.

	Therefore, it remains to obtain the estimate \eqref{E:LPD-1}. We first rewrite
	\[\l|[e^{it\Delta^2}][e^{i(\cdot)\xi_n}\phi](x)\r|=\frac{1}{(2\pi)^2} \l|\int_{\bR^2} e^{ix\xi +it{\Phi}_n(\xi)} \widehat{\phi}(\xi)\ddd\xi\r|\]
	with
	\[{\Phi}_n(\xi):= |\xi|^4+4|\xi|^2\xi\xi_n +2|\xi|^2|\xi_n|^2 +4|\xi\xi_n|^2 +4|\xi_n|^2\xi_n\xi +|\xi_n|^4.\]
	Notice that by density we can assume $\phi$ to be a Schwartz function with compact Fourier support. Then by the change of variables
	\[X=x+4|\xi_n|^2\xi_n t, \quad T= |\xi_n|^2t,\]
	we know that the left hand side of \eqref{E:LPD-1} is comparable to
	\[\lim_{n\to\infty} |\xi_n|^{-\frac{1}{3}} \l\|\int_{\bR^2} e^{i X\xi-i T\l[2+4\cos^2{\theta_{\xi}^n}\r]|\xi|^2-i T\l(4\frac{|\xi|^2\xi\cos\theta_{\xi}^n}{|\xi_n|}+\frac{|\xi|^4}{|\xi_n|^2}\r)} \widehat{\phi}(\xi) \ddd\xi \r\|_{L_{T,X}^{6}},\]
	where
	\[\cos\theta_{\xi}^n:=\frac{\xi\bar{\xi}_n}{|\xi|}, \quad \bar{\xi}_n:= \frac{\xi_n}{|\xi_n|}.\]
	If we further denote
	\[\cos\theta_{\xi}:=\frac{\xi\bar{\xi}_0}{|\xi|}, \quad \bar{\xi}_0:= \lim_{n\to\infty} \bar{\xi}_n,\]
	we claim (the proof is post-posed) that
	\begin{align*}
		\lim_{n\to\infty}&\l\|\int_{\bR^2} e^{i X\xi-i T\l[2+4\cos^2\theta_{\xi}^n\r]|\xi|^2-i T\l(4\frac{|\xi|^2\xi\cos\theta_{\xi}^n}{|\xi_n|}+\frac{|\xi|^4}{|\xi_n|^2}\r)} \widehat{\phi}(\xi) \ddd\xi \r\|_{L_{T,X}^{6}}\\
		=&\l\|\int_{\bR^2} e^{i X\xi-i T(2+4\cos^2{\theta}_{\xi})|\xi|^2} \widehat{\phi}(\xi) \ddd\xi \r\|_{L_{T,X}^{6}}.
	\end{align*}
	Then, due to the condition $\lim_{n\to\infty}|\xi_n|=\infty$, our desired conclusion \eqref{E:LPD-1} can be deduced from the following estimate:
	\begin{equation}\label{E:LPD-2}
		\l\|\int_{\bR^2} e^{i X\xi-i T(2+4\cos^2{\theta_{\xi}})|\xi|^2} \widehat{\phi}(\xi) \ddd\xi \r\|_{L_{T,X}^{6}}\lesssim \|\phi\|_{L^2}.
	\end{equation}
	To establish \eqref{E:LPD-2}, we define the transformation $A_0$ on $\bR^2$ as
        \begin{equation}\label{E:LPD-2.5}
		A_0: \xi\mapsto \sqrt{2}\xi^{\bot}+\sqrt{6}\xi^{\shortparallel}, \quad \xi^{\shortparallel}:=(\xi\xi_0)\xi_0, \quad \xi^{\bot}:=\xi-\xi^{\shortparallel}.
	\end{equation}
	And the associated unitary operator $[\tilde{A}_0]$ on $L^2(\bR^2)$ is defined by
	\[[\tilde{A}_0]f(x):=|A_0|^{1/2}f(A_0x).\]
	Note that $|A_0|=2\sqrt{3}$ and $|A_0\xi|^2=(2+4\cos^2\theta_{\xi})|\xi|^2$.
	Hence a direct computation yields
	\[\l\|\int_{\bR^2} e^{i X\xi-i T(2+4\cos^2{\theta_{\xi}})|\xi|^2} \widehat{\phi}(\xi) \ddd\xi \r\|_{L_{T,X}^{6}}=|A_0|^{-1/3}\l\|[e^{it\Delta}][\tilde{A}_0]\phi\r\|_{L_{T,X}^6}.\]
	Then the classical Strichartz estimates for $[e^{it\Delta}]$ gives the desired result \eqref{E:LPD-2}. For more on this estimate, we refer to \cite{Strichartz1977}.
 
    Now, it remains to prove the aforementioned claim. Indeed, define the phase function
	\[\varphi_n(T,X,\xi):= X\xi- T\l[2+4\cos^2\theta_{\xi}^n\r]|\xi|^2- T\l(4\frac{|\xi|^2\xi\cos\theta_{\xi}^n}{|\xi_n|}+\frac{|\xi|^4}{|\xi_n|^2}\r).\]
	Since the support of $\widehat{\phi}$ is compact and $\lim_{n\to\infty}|\xi_n|=\infty$, we can choose $n$ large enough such that
	\[\l|\l(\frac{\partial}{\partial \xi_1}\r)^2 \varphi_n(T,X,\xi)\r|> |T|\l[1+2\cos^2\theta_{\xi}^n\r]\geq |T|, \quad \forall \;\xi\in\supp\widehat{\phi}.\]
	Furthermore, we can choose $n$ large enough such that the following estimate
	\[\l|\partial_{\xi}^{\beta} \varphi_n(T,X,\xi)\r| \leq |T|/2\]
	holds for arbitrary multi-index $\beta$ with $|\beta|=3$ on the support of $\widehat{\phi}$. Hence, for large $n$, the several variables' scaling principle of oscillatory integrals \cite[page 342, Proposition 5]{Stein1993} gives
	\begin{equation}\label{E:LPD-3}
		\l|\int_{\bR^2} e^{i\varphi_n(T,X,\xi)} \widehat{\phi}(\xi) \ddd\xi \r|\lesssim_{\phi} (1+|T|)^{-1/2}.
	\end{equation}
	This estimate holds uniformly for $X\in \bR^2$. Then we consider the gradient and use the localization principle for several variables' oscillatory integrals \cite[page 341, Proposition 4]{Stein1993}. It is sufficient to consider the partial derivative $\partial_{\xi_1}\varphi_n(T,X,\xi)$. Since $\supp(\widehat{\phi})$ is compact and $\lim_{n\to\infty}|\xi_n|=\infty$, we know that if $|X|\gtrsim_{\phi} |T|$ and $n$ large enough, there holds
	\[\l|\partial_{\xi_1}\varphi_n(T,X,\xi)\r|>|X|/2.\]
	Therefore, on the set $\l\{(X,T)\in\bR^3: |X|\gtrsim_{\phi} |T|\r\}$, the localization principle for several variables' oscillatory integrals \cite[page 341, Proposition 4]{Stein1993} implies
	\[\l|\int_{\bR^2} e^{i\varphi_n(T,X,\xi)} \widehat{\phi}(\xi) \ddd\xi \r|\lesssim_{\phi} (1+|X|)^{-1} \lesssim_{\phi} \l[(1+|X|)(1+|T|)\r]^{-1/2}.\]
	On the other hand, for the set $\l\{(X,T)\in\bR^3: |X|\lesssim_{\phi} |T|\r\}$, we can use \eqref{E:LPD-3} to deduce
	\[\l|\int_{\bR^2} e^{i\varphi_n(T,X,\xi)} \widehat{\phi}(\xi) \ddd\xi \r|\lesssim_{\phi} (1+|T|)^{-1/2} \lesssim_{\phi} \l[(1+|X|)(1+|T|)\r]^{-1/4}.\]
	These aforementioned two estimates give the following dominating function:
	\begin{equation*}
		F(T,X):=\l\{\begin{array}{cc}
			C_{\phi}\l[(1+|T|)(1+|X|)\r]^{-1/4}, & |X|\leq C'_{\phi} |T|;\\
			C_{\phi}\l[(1+|T|)(1+|X|)\r]^{-1/2}, & |X|\geq C'_{\phi} |T|.
			\end{array}\r.
	\end{equation*}
	It is routine to verify that $F(T,X)\in L_{T,X}^6(\bR^3)$ and then the dominated convergence theorem directly implies our claim.
\end{proof}

There are standard arguments to deduce some extremizer results by applying some linear profile decomposition consequences.  See, for instance, \cite{BS2021,Han2015,HS2012,JSS2017,Shao2009EJDE,Stovall2020} for further details. Since there is no frequency parameters in the aforementioned non-endpoint profile decomposition Theorem \ref{T:LPD}, the existence of extremizers is much easier to obtain. Therefore, we state the following Theorem \ref{T:existence of extremizers} without proof. Similar proof can be found in Han's work \cite[Section 3]{Han2015} and the details are omitted here.
\begin{thm}\label{T:existence of extremizers}
	There exists an extremizer for $\mathbf{S}$.
\end{thm}

\section{Differentiability}\label{S:Differentiability}
We investigate the differentiability of extremizers by investigating the corresponding Euler-Lagrange equation. Firstly, for functions $f_i\in L^2(\bR^2)$ with $i\in \{1,\ldots, 6\}$, we introduce a $6$-multilinear form $Q$ as follows:
\[Q(f_1,f_2,f_3,f_4,f_5,f_6):= \int_{\bR^3} \prod_{k=1}^3 \overline{[e^{it\Delta^2}]f_k(x)} \ [e^{it\Delta^2}]f_{k+3}(x) \ddd x\ddd t.\]
A function $f\in L^2(\bR^2)$ is said to be a solution to the associated Euler-Lagrange equation if it satisfies the following equation
\begin{equation} \label{E:Euler-Lagrange}
        \omega \langle f,g \rangle_{L^2} =Q(g, f, f, f, f, f), \quad \forall g\in L^2(\bR^2)
\end{equation}
for some $\omega>0$. Here we have used the notation $\langle f,g\rangle_{L^2}:=\int_{\bR^2} \bar{f}(x) g(x) \ddd x$.  Similar derivations can be found in \cite[page 473]{Evans2010} or \cite[Section 1.2]{HL2012}.  It can be seen that extremizers to \eqref{E:4-Strichartz} are solutions to the associated Euler-Lagrange equation with $\omega=\mathbf{S}^6\|f\|_{L^2}^4$.

Let $\delta$ denote the Dirac delta measure. Note that $\widehat{\delta}\sim 1$ as a distribution. Therefore, we have
\begin{equation} \label{E:6-multilinear sim delta}
	Q(f_1,\ldots, f_6) \sim \int_{\bR^{12}} \overline{\widehat{f}_1}(\tilde{\xi}_1) \overline{\widehat{f}_2}(\tilde{\xi}_2) \overline{\widehat{f}_3}(\tilde{\xi}_3) \widehat{f_4}(\tilde{\xi}_4) \widehat{f_5}(\tilde{\xi}_5) \widehat{f_6}(\tilde{\xi}_6) \delta(a(\tilde{\xi})) \delta(b(\tilde{\xi})) \ddd \tilde{\xi},
\end{equation}
where $\tilde{\xi}_i \in \bR^2$ for $i\in\{1,\ldots, 6\}$ and $\tilde{\xi} :=(\tilde{\xi}_1, ..., \tilde{\xi}_6) \in \left(\bR^2\right)^6$, as well as
\[a(\tilde{\xi}):= \tilde{\xi}_1+\tilde{\xi}_2+\tilde{\xi}_3 - \tilde{\xi}_4 - \tilde{\xi}_5 - \tilde{\xi}_6, \quad b(\tilde{\xi}) :=|\tilde{\xi}_1|^4+|\tilde{\xi}_2|^4 +|\tilde{\xi}_3|^4 -|\tilde{\xi}_4|^4-|\tilde{\xi}_5|^4 - |\tilde{\xi}_6|^4.\]

In order to show differentiability, we follow similar arguments to those found in  \cite{FS2022DPDE,HS2012,JS2016}.  It suffices to prove the following theorem:
\begin{thm} \label{exp f L2}
	If $f$ solves equation \eqref{E:Euler-Lagrange} for some $\omega>0$, then there exists $\mu>0$ such that
	$$e^{\mu|\xi|^4}\widehat{f} \in L^2(\bR^2).$$
	Furthermore, $f$ can be extended to be complex analytic on $\bC^2$.
\end{thm}

Note that the last part follows directly from the Fourier inversion formula and dominated convergence. This is because if there holds $e^{\mu |\xi|^4}\widehat{f} \in L^2$ for some $\mu>0$, then for any $a\in \bR$ we can rewrite $e^{a|\xi|}\widehat{f}$ as
\begin{equation} \label{E:Extend complex analytic}
	e^{a|\xi|}\widehat{f}=e^{a|\xi|-\mu|\xi|^4}e^{\mu |\xi|^4}\widehat{f},
\end{equation}
and, furthermore, the Cauchy-Schwarz inequality implies $e^{a|\xi|}\widehat{f}(\xi)\in L^1(\bR^2)$. Hence, for any $z\in \bC^2$ we could choose $a>|z|$ and conclude $\partial_{\overline{z}} f(x)=0$.
This result also follows directly from the Paley-Wiener theorem \cite[Theorem IX.13]{RS1975}. In fact, notice that the term $e^{a|\xi|-\mu|\xi|^4}$ in the identity \eqref{E:Extend complex analytic} is clearly bounded. Hence, we obtain $e^{a|\xi|}\widehat{f} \in L^2(\bR^2)$ and then the Paley-Wiener theorem implies the desired property.

Meanwhile, we introduce the weighted $6$-linear form $M_F$ as follows:
\begin{equation} \label{E:WeightMLI}
        M_F(h_1, ..., h_6) := \int e^{F(\tilde{\eta}_1)-\sum_{k=2}^6 F(\tilde{\eta}_k)} \prod_{k=1}^6 |h_k(\tilde{\eta}_k)| \delta(a(\tilde{\eta})) \delta(b(\tilde{\eta})) \ddd \tilde{\eta},
\end{equation}
where $\tilde{\eta}_i \in \bR^2$ for $i\in \{1,\ldots, 6\}$ and $\tilde{\eta}:=(\tilde{\eta}_1, ..., \tilde{\eta}_6) \in \left(\bR^2\right)^6$, as well as
\begin{equation} \label{E:Weight function}
    F(\tilde{\eta}_k)=F_{\mu, \eps }(\tilde{\eta}_k) := \frac{\mu |\tilde{\eta}_k|^4}{1+ \eps |\tilde{\eta}_k|^4}, \qquad \eps \geq 0, \ \mu \geq 0.
\end{equation}

Notice that $F(\tilde{\eta}_k)$ increases with respect to $|\tilde{\eta}|$. Indeed, by considering the function $\frac{\mu x^4}{1+\eps x^4}$, we can see that the derivative of this function is strictly positive for $x>0$. In the support dictated by $\delta(b(\tilde{\eta}))$, there holds
$$|\tilde{\eta}_1|^4+|\tilde{\eta}_2|^4+|\tilde{\eta}_3|^4=|\tilde{\eta}_4|^4+|\tilde{\eta}_5|^4+|\tilde{\eta}_6|^4.$$
Hence we conclude
\[F(\tilde{\eta}_1) \leq \sum_{k=2}^6 F(\tilde{\eta}_k).\]
This immediately gives the inequality
\begin{equation*}
       \left| M_F(h_1, ..., h_6)\right|\leq \int_{\bR^{12}} \prod_{k=1}^6 \left| h(\tilde{\eta}_k)\right| \delta(a(\tilde{\eta})) \delta(b(\tilde{\eta})) \ddd\tilde{\eta}.
\end{equation*}
Thus, the previous estimate \eqref{E:6-multilinear sim delta} directly implies that
\begin{equation}\label{E:M_F increasing}
    \left| M_F(h_1, ..., h_6)\right| \lesssim Q(|h_1|^{\vee},\ldots |h_6|^{\vee}).
\end{equation}
To establish the desired Theorem \ref{exp f L2}, one crucial tool is the following bilinear estimate lemma.
\begin{lemma} \label{bilinear}
	Let $s \gg 1$ and $N \gg 1$. For the functions $f\in L^2(\bR^2)$ and $g\in L^2(\bR^2)$, if their Fourier supports satisfy
	\[\supp{\widehat{f}}\subset \{\xi\in \bR^2: |\xi|\leq s\}, \quad \supp{\widehat{g}}\subset \{\eta\in \bR^2: |\eta|\geq Ns\},\]
	then there holds
	\[\left\|[e^{it\Delta^2}]f \ [e^{it\Delta^2}]g \right\|_{L_{t,x}^3(\bR^3)} \lesssim N^{-5/6} \|f\|_{L_x^2(\bR^2)}\|g\|_{L_x^2(\bR^2)}.\]
\end{lemma}

\begin{proof}[\textbf{Proof of Lemma \ref{bilinear}}]
We take some inspiration from Killip and Visan \cite[pages 366-367]{KV2013}. Firstly, by breaking the region of integration into several pieces and rotating the coordinate system appropriately, we may restrict our view to where $\eta_1-\xi_1 \gtrsim Ns$. Then consider the following change of variables:
\begin{equation*}
    \gamma=\xi+\eta, \quad \tau=|\xi|^4+|\eta|^4, \quad \beta=\xi_2.
\end{equation*}
A direct computation shows
\begin{align}
       e^{it\Delta^2}f(x) e^{it\Delta^2}g(x) &=\frac{1}{(2\pi)^4}\int_{\bR^4} e^{i(x\cdot \gamma + t \tau)} \ \widehat{f}(\gamma, \tau, \beta) \ \widehat{g}(\gamma, \tau, \beta) \ |\det J| \ddd\gamma \ddd \tau \ddd \beta. \label{E:Bilinear-1}
\end{align}
The absolute value of the determinant of our Jacobian matrix is $4\left|\eta_1|\eta|^2-\xi_1|\xi|^2\right|$.
Recall the previous assumptions $\eta_1-\xi_1 \gtrsim Ns$ and $|\xi|\leq s$. Therefore, in the dyadic region $2^kNs \leq |\eta| \leq 2^{k+1}Ns$, there holds
\begin{equation}\label{E:bilinear-1}
	\l|\det J^{-1} \r|\gtrsim 2^{2k}(Ns)^3.
\end{equation}
Consider the function $G_\beta (\gamma, \tau):=\widehat{f}(\gamma,\tau, \beta)\widehat{g}(\gamma,\tau, \beta)|\det J|$ and recall the inverse space-time Fourier transform
\begin{equation*}
	\widetilde{F}(\xi, \tau):= \int_{\bR^3} e^{i(x\cdot \xi+t\tau)}F(x, t) \ddd x\ddd t.
\end{equation*}
Then we have
\[\widetilde{G}_\beta (x, t)=\int_{\bR^3} e^{i(x\cdot\gamma+ t\tau)} \widehat{f}(\gamma, \tau, \beta) \widehat{g}(\gamma, \tau, \beta) |\det J| \ddd\gamma \ddd\tau.\]
By the identity \eqref{E:Bilinear-1}, it is clear that
\[e^{it\Delta^2}f(x) e^{it\Delta^2}g(x)=\int_{\bR} \widetilde{G}_\beta (x, t) \ddd \beta.\]
Applying Minkowski's inequality followed by Hausdorff-Young's inequality gives
\[\left\| e^{it\Delta^2}f e^{it\Delta^2}g \right\|_{L_{t,x}^3(\bR^3)} \leq \int_{\bR} \|\widetilde{G}_\beta (x, t)\|_{L_{t,x}^3(\bR^3)} \ddd \beta \leq \int_{\bR} \left\|G_\beta (\gamma, \tau) \right\|_{L_{\gamma,\tau}^{3/2} (\bR^3)} \ddd\beta.\]
Notice that $|\beta| \leq s$. Thus, using Hölder's inequality for $p=3$ and $q=3/2$ we conclude 
\begin{equation*}
    \begin{split}
        \left\| e^{it\Delta^2}f e^{it\Delta^2}g \right\|_{L_{t,x}^3(\bR^3)} &\leq \int_{\bR} \left\|G_\beta (\gamma, \tau) \right\|_{L_{\gamma,\tau}^{3/2} (\bR^3)} \ddd\beta \\
        &\leq \left(\int_{|\beta|\leq s} 1 \ddd \beta \right)^{1/3} \left(\int_{\bR^4} |G_\beta(\gamma, \tau)|^{3/2} \ddd\gamma \ddd\tau \ddd\beta \right)^{2/3} \\
        &\leq s^{1/3} \left\|G_\beta(\gamma, \tau)\right\|_{L_{\gamma,\tau,\beta}^{3/2}(\bR^4)}.
    \end{split}
\end{equation*}
Recall that $G=\widehat{f}\widehat{g}|\det J|$. Then changing the variables back to $\xi$ and $\eta$ implies
\[\left\| e^{it\Delta^2}f e^{it\Delta^2}g \right\|_{L_{t,x}^3(\bR^3)} \leq s^{1/3}\left(\int_{\bR^4} \left|\widehat{f}(\xi)\right|^{3/2}\left|\widehat{g}(\eta)\right|^{3/2}|J|^{1/2} \ddd\xi\ddd\eta \right)^{2/3}.\]
Since $\frac{2}{3}<1$, by the dyadic decomposition we can deduce the following estimates:
\begin{align*}
        \left\|e^{it\Delta^2}f e^{it\Delta^2}g \right\|_{L_{t,x}^3(\bR^3)} &\leq s^{1/3} \left(\sum_{k=0}^\infty \int_{|\xi|\leq s} \int_{2^kNs\leq |\eta|\leq 2^{k+1}Ns}\left|\widehat{f}(\xi)\right|^{3/2}\left|\widehat{g}(\eta)\right|^{3/2}|J|^{1/2} \ddd\eta\ddd\xi\right)^{2/3} \\
        &\leq s^{1/3}\sum_{k=0}^\infty \left( \int_{|\xi|\leq s} \int_{2^kNs\leq |\eta|\leq 2^{k+1}Ns}\left|\widehat{f}(\xi)\right|^{3/2}\left|\widehat{g}(\eta)\right|^{3/2}|J|^{1/2} \ddd\eta\ddd\xi\right)^{2/3}.
\end{align*}
Then \eqref{E:bilinear-1} further gives
\begin{align}
        \left\|e^{it\Delta^2}f e^{it\Delta^2}g \right\|_{L_{t,x}^3(\bR^3)} &\lesssim s^{\frac{1}{3}} \sum_{k=0}^\infty \left( \int_{|\xi|\leq s} \int_{2^kNs\leq |\eta|\leq 2^{k+1}Ns} \left|\widehat{f}(\xi)\right|^{\frac{3}{2}} \left| \widehat{g}(\eta)\right|^{\frac{3}{2}} 2^{-k}(Ns)^{-\frac{3}{2}} \ddd\eta\ddd\xi \right)^{\frac{2}{3}} \notag\\
        &=N^{-1}s^{-\frac{2}{3}} \sum_{k=0}^\infty 2^{-\frac{2k}{3}} \left( \int_{|\xi|\leq s} \int_{2^k Ns\leq |\eta|\leq 2^{k+1} Ns} \left|\widehat{f}(\xi)\right|^{\frac{3}{2}} \left|\widehat{g}(\eta)\right|^{\frac{3}{2}} \ddd\eta\ddd\xi \right)^{\frac{2}{3}} \notag\\
        &=N^{-1}s^{-\frac{2}{3}}\left(\int_{|\xi|\leq s}|\widehat{f}(\xi)|^{\frac{3}{2}} \ddd\xi \right)^{\frac{2}{3}} \sum_{k=0}^\infty 2^{\frac{-2k}{3}} \left(\int_{2^kNs\leq |\eta|\leq 2^{k+1}Ns} \!|\widehat{g}(\eta)|^{\frac{3}{2}} \ddd \eta \right)^{\frac{2}{3}}\!\!. \label{E:bilinear-2}
\end{align}
Using Hölder's inequality for $p=4/3$ and $q=4$, we conclude that
\begin{equation*}
        \int_{|\xi|\leq s} |\widehat{f}(\xi)|^{3/2} \ddd\xi \leq \left(\int_{|\xi|\leq s}1^4 \ddd\xi \right)^{1/4} \left(\int_{|\xi|\leq s}|\widehat{f}|^2 \ddd\xi \right)^{3/4} \lesssim s^{1/4}\|f\|_{L^2}^{3/2}.
\end{equation*}
Similar logic also gives us that 
$$\int_{2^kNs\leq |\eta|\leq 2^{k+1}Ns} |\widehat{g}(\eta)|^{3/2} \ddd\eta \lesssim 2^{k/4}(Ns)^{1/4}\|g\|_{L^2}^{3/2}.$$
Therefore, inserting these two estimates into \eqref{E:bilinear-2}, we can deduce
\begin{align*}
        \left\| e^{it\Delta^2}f e^{it\Delta^2}g \right\|_{L_{t,x}^3(\bR^3)}&\lesssim N^{-1}s^{-2/3}s^{1/6} \|f\|_{L^2} \sum_{k=0}^\infty 2^{-2k/3}2^{k/6}(Ns)^{1/6} \|g\|_{L^2} \\
        &\lesssim N^{-5/6} \|f\|_{L^2} \|g\|_{L^2},
\end{align*}
where in the last inequality we have used the condition $s\gg1$. And now the proof is completed.
\end{proof}

We note that this Lemma \ref{bilinear} can be seen in another light. One can follow Bourgain's argument in \cite[Lemma 111]{Bourgain1998} to establish a similar bilinear estimate. Here we omit the details for simplicity.

Using estimate \eqref{E:M_F increasing} and then using the $L_x^2 \rightarrow L_{t,x}^6$ Strichartz inequality \eqref{E:4-Strichartz}, as well as Hölder's inequality, we can directly deduce the following multilinear estimate Lemma \ref{multilinear} based on the aforementioned bilinear estimate Lemma \ref{bilinear}. Similar arguments can also be found in previous works such as \cite[estimate (5-5)]{BOQ2020} and \cite[Corollary 4.7]{HS2012}. For simplicity, the detailed proof is omitted here. 

\begin{lemma} \label{multilinear}
	Set $s\gg 1$ and $N\gg 1$. Assume that $h_1$ and $h_2$ have Fourier supports satisfying
	\[\supp{\widehat{h_1}}\subset \{\xi: |\xi|\leq s\}, \quad \supp{\widehat{h_2}} \subset \{\eta: |\eta|\geq Ns\}.\]
	If the functions $h_i \in L^2(\bR^2)$ for all $i\in\{1,2\ldots,6\}$, then
	\[M_F(h_1, ..., h_6) \lesssim N^{-5/6} \prod_{k=1}^6 \|h_k\|_{L^2}.\]
\end{lemma}

We now provide the following Lemma \ref{l:Poly} to help establish our Theorem \ref{exp f L2}. The proof of this lemma is nearly identical to the proof of Lemma $3.3$ in \cite{FS2022DPDE}, and takes inspiration from the proof of Lemma $2.2$ in \cite{JS2016}, as well as Proposition $5.2$ in \cite{BOQ2020}. These are slightly different, given the order of the problem as well as the dimension. The only substantial difference between the proof for this and the proof in \cite{FS2022DPDE} is for our proof we let $\mu=s^{-8}$, and in \cite{FS2022DPDE} it is $\mu=s^{-4}$. We will provide the details below.

\begin{lemma} \label{l:Poly}
	Assume $f$ solves the generalized Euler-Lagrange equation \eqref{E:Euler-Lagrange} and $\|f\|_{L^2(\bR^2)}=1$.  Furthermore, define the notation $\widehat{f}_>:=\widehat{f}\mathds{1}_{|\xi|\geq s^2}$ for $s>0$ and recall the function $F=F_{\mu,\eps}$ defined in \eqref{E:Weight function}. Then there exists some $s \gg 1$ such that for $\mu=s^{-8}$, we have
	\begin{equation*}
    	\begin{split}
        	\omega \left\|e^{F(\cdot)}\widehat{f}_>\right\|_{L^2}&\leq o_1(1)\left\|e^{F(\cdot)}\widehat{f}_>\right\|_{L^2}+C\left\|e^{F(\cdot)}\widehat{f}_>\right\|_{L^2}^2 +C\left\|e^{F(\cdot)}\widehat{f}_> \right\|_{L^2}^3 \\
        	&+C\left\|e^{F(\cdot)}\widehat{f}_>\right\|_{L^2}^4 +C\left\|e^{F(\cdot)}\widehat{f}_>\right\|_{L^2}^5+o_2(1),
    	\end{split}
	\end{equation*}
	where $o_k(1)$ means that the constant $o_k(1)\to 0$ as $s\to \infty$ uniformly for all $\eps>0$ and $k=1,2$.  Likewise, the constant $C$ is independent of $\eps$ and $s$.
\end{lemma}

\begin{proof}[\textbf{Proof of Lemma \ref{l:Poly}}]
We start by noting that
\[\left\|e^{F(\cdot)}\widehat{f}_>\right\|_{L^2}^2 =\l\langle e^{2F(\cdot)}\widehat{f}_>, \widehat{f}_> \r\rangle_{L^2} =\l\langle e^{2F(\cdot)}\widehat{f}_>, \widehat{f}\r\rangle_{L^2} =\l\langle [e^{2F}]f_>, f\r\rangle_{L^2},\]
where $[e^{2F}]f_>:=\mathscr{F}^{-1} e^{2F(\xi)} \mathscr{F} f_>$. Since $f$ solves the equation \eqref{E:Euler-Lagrange}, we have
\[\omega  \left\|e^{F(\cdot)}\widehat{f}_> \right\|_{L^2}^2 =Q\l([e^{2F}]f_>, f, f, f, f, f\r).\]
Therefore, if we define the following functions on Fourier space
\[h(\tilde{\xi}_i) :=e^{F(\tilde{\xi}_i)}\widehat{f}(\tilde{\xi}_i), \quad h_>(\tilde{\xi}_i) :=e^{F(\tilde{\xi}_i)}\widehat{f}_>(\tilde{\xi}_i)\]
for $i\in \{1,\ldots, 6\}$ and $\tilde{\xi}=(\tilde{\xi}_1, ..., \tilde{\xi}_6) \in \left(\bR^2\right)^6$, then \eqref{E:6-multilinear sim delta} implies that
\begin{align*}
\omega \left\|e^{F(\cdot)}\widehat{f}_>\right\|_{L^2}^2 &\sim \int_{\bR^{12}} e^{F(\tilde{\xi}_1) -\sum_{k=2}^6F(\tilde{\xi}_k)} h_>(\tilde{\xi}_1) \prod_{i=2}^6 h(\tilde{\xi}_i) \delta(a(\tilde{\xi})) \delta(b(\tilde{\xi})) \ddd\tilde{\xi} \\
&=M_F(h_>, h, h, h, h, h).
\end{align*}
To establish our desired result, it suffices to bound $M_F(h_>, h, h, h, h, h)$. Now let us introduce the notations
\[h_\ll :=h\mathbbm{1}_{\{|\xi|<s\}}, \quad h_\sim :=h\mathbbm{1}_{\{s\leq |\xi|\leq s^2\}}.\]
Then there hold $h=h_>+h_<$ and $h_<=h_\ll + h_\sim$.  Hence we can write
\begin{equation*}
    \begin{split}
        M_F(h_>, h, h, h, h, h)&=M_F(h_>, h_<, ..., h_<)+\sum_{j_2, ..., j_6} M_F(h_>, h_{j_2}, ..., h_{j_6}) \\
        &=:A+B,
    \end{split}
\end{equation*}
where $j_i$ is either $``<"$ or $``>"$, and there is at least one $``<"$ in the set $\{j_2, ..., j_6\}$.  

We first deal with the term $A$ by using Lemma \ref{multilinear}. Note that we can further decompose $A$ as
\begin{equation*}
    \begin{split}
        A&=M_F(h_>, h_\ll, h_<, ..., h_<)+M_F(h_>, h_\sim, h_<, ..., h_<) \\
        &=:A_1+A_2.
    \end{split}
\end{equation*}
Hence, for the term $A_1$, Lemma \ref{multilinear} gives the following estimate
\[A_1\lesssim s^{-5/6} \left\|h_>\right\|_{L^2}\left\|h_\ll\right\|_{L^2} \left\|h_<\right\|_{L^2}^4.\]
Furthermore, for $\|h_<\|_{L^2}$, we have
\begin{equation*}
        \left\|h_<\right\|_{L^2} =\l(\int_{\bR^2} e^{\frac{2\mu |\xi|^4}{1+\eps |\xi|^4}} \left|\widehat{f}(\xi)\right|^2 \mathbbm{1}_{\{|\xi|<s^2\}} \ddd\xi\r)^{1/2} \leq e^{\mu s^8}\left\|f\right\|_{L^2} =e^{\mu s^8}.
\end{equation*}
Following similar logic gives us that
\[\left\|h_\ll \right\|_{L^2} \leq e^{\mu s^4}, \quad \left\|h_\sim \right\|_{L^2} \leq e^{\mu s^8} \left\|f_\sim \right\|_{L^2},\]
where $\widehat{f}_\sim :=\widehat{f}\mathbbm{1}_{s\leq |\xi|\leq s^2}$. For the term $A_2$, based on \eqref{E:M_F increasing} and the Strichartz inequality \eqref{E:4-Strichartz}, we can directly use \holder inequality to obtain
\[A_2\lesssim \left\|h_>\right\|_{L^2} \left\|h_\sim \right\|_{L^2} \left\|h_<\right\|_{L^2}^4.\]
Note that there holds $\left\|f_\sim \right\|_{L^2} \to 0$ as $s \to\infty$. Hence by combining all the aforementioned estimates and letting $\mu=s^{-8}$, we can deduce
\begin{equation*}
    \begin{split}
        A &\lesssim s^{-5/6} \left\|h_>\right\|_{L^2} \left\|h_\ll\right\|_{L^2} \left\|h_<\right\|_{L^2}^4 +\left\|h_>\right\|_{L^2} \left\|h_\sim \right\|_{L^2} \left\|h_<\right\|_{L^2}^4 \\
        &\leq s^{-5/6} \left\|h_>\right\|_{L^2} e^{\mu s^4} e^{4\mu s^8}+\left\|h_>\right\|_{L^2} e^{\mu s^8} \left\|f_\sim \right\|_{L^2} e^{4\mu s^8} \\
        &=e^{5}\left\|h_>\right\|_{L^2} \left(s^{-5/6}e^{s^{-4}-1}+ \left\|f_\sim\right\|_{L^2}\right) \\
        &=o_1(1) \left\|e^{F(\cdot)} \widehat{f}_> \right\|_{L^2}.
    \end{split}
\end{equation*}
Here the constant $o_1(1)$ is obviously independent of $\eps$, and converges to $0$ as $s$ gets large.  

Now let us turn to the term $B$. We break up $B$ based on how many subscripts are $``>"$. In other words, for $k\in \{1,\ldots,5\}$, we define
\[B_k:=\sum_{j_2, ..., j_6} M_F(h_>, h_{j_2}, ..., h_{j_6}),\]
where precisely $k$ terms in each set $\{h_{j_2}, ..., h_{j_6}\}$ are $h_>$.  
Indeed, we can rearrange each item so that there holds $B_1=5 M_F(h_>, h_>, h_<, h_<, h_<, h_<)$.
Following a similar argument that we made for $A$, we can bound $B_1$ by
\[B_1 \lesssim o_2(1) \left\|e^{F(\cdot)} \widehat{f}_>\right\|_{L^2}^2,\]
where $o_2(1)$ approaches $0$ as $s$ gets large. Furthermore, for $k=\{2,\ldots,5\}$, we can do the same argument to get
\[B_k \lesssim \left\|e^{F(\cdot)}\widehat{f}_>\right\|_{L^2}^{k+1}.\]
Hence we finally obtain the following estimate
\begin{align*}
        \omega \left\| e^{F(\cdot)}\widehat{f}_> \right\|_{L^2}^2 & \leq o_1(1) \left\| e^{F(\cdot)} \widehat{f}_> \right\|_{L^2} +o_2(1)\left\| e^{F(\cdot)}\widehat{f}_> \right\|_{L^2}^2 +C\left\| e^{F(\cdot)}\widehat{f}_> \right\|_{L^2}^3 \\
        &+C\left\| e^{F(\cdot)}\widehat{f}_> \right\|_{L^2}^4
        +C\left\| e^{F(\cdot)}\widehat{f}_> \right\|_{L^2}^5 +C\left\| e^{F(\cdot)}\widehat{f}_> \right\|_{L^2}^6.
\end{align*}
Dividing both sides by $\left\|e^{F(\cdot)}\widehat{f}_>\right\|_{L^2}$, we get our desired result and the proof is completed.
\end{proof}

We now show that $e^{\mu|\xi|^4}\widehat{f} \in L^2(\bR^2)$ for some $\mu>0$ if the function $f$ is an extremizer to \eqref{E:4-Strichartz}. The same logic that is applied in \cite[Theorem 1.5]{BOQ2020} and \cite[Theorem 1.1]{HS2012} works here, and we provide it here for completeness.

\begin{proof}[\textbf{Proof of Theorem \ref{exp f L2}}]
Without loss of generality we may assume 
$\|f\|_{L^2}=1$ at the beginning. We start by defining 
\[\widehat{f}_{>}:=\widehat{f} \mathbbm{1}_{\{|\xi|\geq s^2\}}, \quad f_{<}:=f-f_{>},\quad \mu:=s^{-8}, \quad H_s(\eps):= \left( \int_{\{|\xi|\geq s^2\}} \left| e^{F_{\mu, \eps }(\xi)}\widehat{f} \right|^2 \ddd \xi \right)^{1/2}.\]
Then on the interval $(0, \infty)$, dominated convergence theorem implies that the function $H_s(\eps)$ is continuous for every fixed $s$, and thus its image is connected since $(0,\infty)$ is connected. Also recall that $f$ solves the Euler-Lagrange equation \eqref{E:Euler-Lagrange} with the constant $\omega=\mathbf{S}^6$. Based on Lemma \ref{l:Poly}, we consider the following function:
\[G(x)=\frac{\mathbf{S}^6}{2}x-C(x^2+x^3+x^4+x^5), \quad x\in [0,\infty).\]
Here the constant $C$ is the same as the corresponding constant in Lemma \ref{l:Poly}, and we define the constant $M:=\sup_{[0,\infty)}G(x)$. Therefore, by Lemma \ref{l:Poly}, we could choose $s$ large enough such that 
\[o_1(1)< \mathbf{S}^6/2,\quad G(H_s(\varepsilon))\leq o_2(1) <M/2\]
hold uniformly for all $\varepsilon>0$. Note that there holds
\[G(0)=0,\quad G'(0)>0, \quad G''(x)<0 \;\; \text{for} \;\; x\in(0,\infty).\]
Hence, we conclude that the equation $G(x)=\frac{M}{2}$ has two different roots on the interval $(0,\infty)$. Denote these two roots $x_0$ and $x_1$ with $0<x_0<x_1$.

In view of the composed function $G(H_s(\varepsilon))$, due to the aforementioned connectivity of $H_s(\varepsilon)$, we know that $G^{-1}([0, M/2])$ must be contained in either $[0, x_0]$ or $[x_1, \infty)$ for every fixed $s$.  However, letting $s$ be sufficiently large and then taking $\eps=1$ would lead to $H_s(1)<x_0$.  Hence, in our situation, there holds
\[G^{-1}([0, M/2])\subset [0, x_0], \quad \text{for} \;\; s\gg1.\]
This fact means that $H_s(\eps)$ is uniformly bounded on $(0, \infty)$ for $s$ large enough. Finally, by using Fatou's lemma or the monotone convergence theorem, we get that $H_s(0)$ is bounded for $s$ large enough. In other words, there holds
\[e^{s^{-8} |\xi|^4}\widehat{f}_{>} \in L^2(\bR^2), \quad \text{for} \;\; s\gg 1.\]
On the other hand, the function $e^{s^{-8}|\xi|^4}\widehat{f}_{<}$ obviously belongs to $L^2(\bR^2)$ for every fixed $s$. Therefore, we obtain our desired result and the proof is finished.
\end{proof}

\bigskip\bigskip\bigskip
\begin{flushleft}
    \textsc{Boning Di\\
    School of Mathematical Sciences \\
    University of Chinese Academy of Sciences \\
    Beijing, 100049 \\ P.R.China} \\
    \textsc{Academy of Mathematics and Systems Science\\ Chinese Academy of Sciences, Beijing, 100190 \\ P.R.China} \\
    E-mail: \textsf{diboning18@mails.ucas.ac.cn}
\end{flushleft}
\begin{flushleft}
    \textsc{Ryan Frier\\
    School of Mathematical and Statistical Sciences \\
    Arizona State University \\
    Tempe, AZ 85281 \\ USA} \\
    E-mail: \textsf{rfrier@asu.edu}
\end{flushleft}

\end{document}